\documentclass[a4paper,11pt]{article}
\usepackage{amsmath,amssymb,amsfonts,graphicx,psfrag,subfig, amsthm}
\usepackage[latin1]{inputenc}

\newtheorem{thrm}{Theorem}[section]
\newtheorem{prpstn}[thrm]{Proposition}
\newtheorem{lmm}[thrm]{Lemma}
\newtheorem{dfntn}[thrm]{Definition}

\newtheorem{rmrk}[thrm]{Remark}

\def\be#1 {\begin{equation} \label{#1}}
\newcommand{\ee}{\end{equation}}

\newcommand{\mb}{\medskip\noindent}

\def \ll {\langle}
\def \rr {\rangle}

\def \I {\{1,\, \dots,p\}}
\def \P {\textsc{P}}
\newcommand{\R}{\mathbb R}

\def \pp {{\mathrm p}}
\def \qqq {{\mathrm{q}}}
\def \qq {\mathrm{q}}

\def \ff {{\bf U}}

\def \ff {\mathrm{f} }

\def \UU {{\bf U}}

\def \NN {\mathrm{N}}
\def \Qc {\tilde{Q}}

\def \virg {\, , \,\,}
\def \dsp {\displaystyle}
\def \vsp {\vspace{6pt}}
\def \noi {\noindent}

\def\sqw{\hbox{\rlap{\leavevmode\raise.3ex\hbox{$\sqcap$}}$%
\sqcup$}}
\def\findem{\ifmmode\sqw\else{\ifhmode\unskip\fi\nobreak\hfil
\penalty50\hskip1em\null\nobreak\hfil\sqw
\parfillskip=0pt\finalhyphendemerits=0\endgraf}\fi}

\title{Convergence order of a numerical scheme for sweeping process}

\date{\today }
\author{Fr\'ed\'eric Bernicot \\ CNRS - Universit\'e de Nantes \\ Laboratoire de Math\'ematiques Jean Leray \\ 2, Rue de la Houssini\`ere F-44322 Nantes Cedex 03 \\
frederic.bernicot@univ-nantes.fr
 \and Juliette Venel\\ LAMAV, Universit\'e de Valenciennes et du Hainaut-Cambr\'esis\\Mont Houy 59313 Valenciennes Cedex 9
\\juliette.venel@univ-valenciennes.fr}

\begin{document}

\maketitle
\tableofcontents

\begin{abstract}
In \cite{jvnm}, an implementable algorithm was introduced to compute discrete solutions of sweeping processes (i.e. specific first order differential inclusions). The convergence of this numerical scheme was proved thanks to compactness arguments. Here we establish that this algorithm is of order $\frac{1}{2}$.
The considered sweeping process involves a set-valued map given by a finite number of inequality constraints. The proof rests on a metric qualification condition between the sets associated to each constraint.
\end{abstract}

\mb {\bf Key-words:} Differential inclusions - Subdifferential calculus - Numerical analysis - Prediction-correction algorithm. 

\mb {\bf MSC:} 34A60, 65L20.

\section{Introduction}

In this paper, we are interested in the discretization of some first order differential inclusions
\be{eq1}  \frac{d\qqq}{dt}(t) + U(t,\qqq(t)) \ni \ff(t,\qqq(t)),\ee
involving a multivalued operator $U$ and a mapping $\ff$.

In the case where $U$ is the subdifferential of the indicatrix of a closed convex set $K$ ($U(\qqq(t))= \partial I_{K}(\qqq(t))$), G. Lippold has shown in \cite{Lippold} that the implicit Euler method converges with an order $h$, where $h$ is the time-step. In the framework of sweeping processes by moving convex sets $K(t)$ (i.e. $U(t,\qqq(t))= \partial I_{K(t)}(\qqq(t))$), J.J. Moreau has proved in \cite{Moreau} the existence of solutions of (\ref{eq1}) with $f=0$. More precisely, he introduced the {\it Catching-up algorithm} to construct discrete solutions converging with an order of precision $\frac{1}{2}$. In \cite{BS}, $U$ takes the form of a maximal monotone operator perturbed by a Lipschitz coercive function. Similarly the authors regain that the order of convergence is $\frac{1}{2}$. In \cite{Colombo}, $U$ is the proximal normal cone to a uniform prox-regular set $Q(t)$ (i.e. $U(t,\qqq(t))= \NN(Q(t),\qqq(t))$ (see Definition~\ref{def:N})). Even if $U$ is not maximal monotone, G. Colombo and V.V. Goncharov also recover the same convergence rate (with $f=0$). These schemes adapted to differential inclusions have given rise to other works, see for instance \cite{Dontchev, Monteiro,  Lempio,  Thibrelax, Thibbv}.

\mb Here, we deal with a perturbed sweeping process:
\be{eq:includiff} \frac{d\qqq}{dt}(t) + \NN(Q(t),\qqq(t)) \ni \ff(t,\qqq(t)),\ee
where $Q(t)$ is defined by a finite number of inequality constraints and known to be uniform prox-regular. We recall that the projection onto such a set is well-defined in its neighbourhood but it cannot be exactly computed. That is why we consider a modified numerical scheme, that is based on a local approximation of $Q(t)$. Near a point $\qqq\in Q(t)$ the set $Q(t)$ is replaced with a closed convex set $\Qc(t,\qqq)$. This substitution makes loose the prescribed regularity of the moving sets (actually these sets are generally supposed to vary in a absolutely continuous way in the sense of the Hausdorff distance). In order to go around this problem, we need to check metric qualification conditions between the sets associated to each constraint. In this context, we keep the same order of convergence:

\begin{thrm} \label{thm:rateconvintro} 
There exists a constant $C_0>0$ such that for $h$ small enough
$$ \|\qqq_h-\qqq\|_{L^\infty([0,T])} \leq C_0 h^{1/2},$$
where $\qqq$ and $\qqq_h$ are the continuous and discrete solutions of (\ref{eq:includiff}).
\end{thrm}

\mb We emphasize that this new approach allows us to avoid resorting to compactness arguments, used in \cite{jvnm}. So it permits to extend the convergence result of \cite{jvnm} in an infinite dimensional Hilbert space.

\bigskip
\noi The paper is structured as follows: In Section~\ref{sec:nota}, we describe the mathematical framework by specifying notations and assumptions which will be used throughout the paper. 
Then in Section~\ref{sec:scheme}, after recalling the prediction-correction scheme proposed in \cite{jvnm}, we prove in Theorem~\ref{thm:rateconv} that the discrete solution converges to the exact solution with order $\frac{1}{2}$. This proof rests on a metric qualification condition which is checked in Section~\ref{sec:mqc}. Finally, we illustrate this result with numerical simulations.

\section{Context}
\label{sec:nota}

 In the sequel, the space $\R^d$ is equipped with its Hilbertian structure. We write $B(x,r)$ for the open ball of center $x\in\R^d$ and radius $r>0$. \\
We consider perturbed sweeping process by a set-valued map $Q:[0,T]\rightrightarrows \R^d$ satisfying that for every $t\in [0,T]$, $Q(t) $ is the intersection of complements of smooth convex sets. 
Let us first specify the set-valued map $Q$.
For $i \in \I$, let $g_i: [0,T] \times \R^d \rightarrow \R$ be a convex function with respect to the second variable. 
For every $t \in [0,T]$, we introduce the sets $Q_i(t)$ defined by
\begin{equation}
 Q_i(t):=\left\{ \qqq \in \R^d \virg g_i(t,\qqq) \geq 0 \right\},
\label{def:Qi}
\end{equation}
and the feasible set $Q(t) $ (supposed to be nonempty) is
\begin{equation}
 Q(t):=\bigcap_{i=1}^{p} Q_i(t).
\label{def:Q}
\end{equation}
Let $\ff: [0,T] \times \R^d\rightarrow \R^d$ be a map, the associated (perturbed) sweeping process can be expressed as follows:
\begin{equation}
\left\{
\begin{array}{l}
\dsp \frac{d\qqq}{dt}(t) + \NN(Q(t),\qqq(t)) \ni \ff(t,\qqq(t)) \ \textmd{for a.e. } t \in [0,T] \vsp \\
\qqq(0)=\qqq_0 \in Q(0).
\end{array}
\right.
\label{eq:incldiff}
\end{equation}

We write $\NN(Q(t),\qqq(t))$ for the proximal normal cone to $Q(t)$ at $\qqq(t)$, defined below.

\begin{dfntn}[\cite{Clarke}] \label{def:N}
Let $S$ be a closed subset of $\R^d $.\\ The proximal
normal cone to $S$ at $x$ is defined by:  
$$ \NN(S,x):= \left \{
v \in \R^d, \  \exists \alpha >0, \  x \in \P_S(x + \alpha
 v) \right \}, $$
where  $$\P_S(y):=\{z \in S, \ d_S(y)= |y -z|\}, \quad \textmd{ with } \quad d_S(y):=\inf_{z \in S} |y- z|  $$
corresponds to the Euclidean projection onto $S$.
\end{dfntn}

\mb We now come to the notion of uniformly prox-regular set. It was initially introduced by H. Federer (in \cite{Federer}) in finite dimensional spaces under the name of ``positively reached set''. Then it was extended in infinite dimensional spaces and studied by F.H.~Clarke, R.J.~Stern and P.R.~Wolenski in \cite{Clarke} and by R.A. Poliquin, R.T.~Rockafellar and L.~Thibault in \cite{PRT}.

\begin{dfntn} Let $S$ be a closed subset of $\R^d$ and $\eta>0$. The set $S$ is said $\eta$-prox-regular if for all $x\in S$ and $v\in \NN(S,x)\setminus \{0\}$
 $$B\left(x+\eta\frac{v}{|v|},\eta\right) \cap S = \emptyset.$$
Equivalently (see \cite{PRT}), $S$ is $\eta$-prox-regular if and only if the projection operator $\P_S$ is single-valued on 
$$ \{z\in\R^d, \ d_S(z)<\eta\}.$$
\end{dfntn}

\noi This differential inclusion can be thought as follows: the point $\qqq(t)$, submitted to the perturbation $\ff(t,\qqq(t))$, has to stay in the feasible set $Q(t)$.
To obtain well-posedness results for (\ref{eq:incldiff}), we will make the following assumptions which ensure the uniform prox-regularity of $Q(t)$ for all $t \in [0,T] $. 
We suppose there exist $c>0$ and open sets $U_i(t) \supset Q_i(t) $ for all $ t$ in $[0,T]$ such that
 \begin{equation}
  \tag{A0}
 d_H(Q_i(t), \R^d \setminus U_i(t)) > c,
 \label{Ui}
\end{equation}
where $d_H $ denotes the Hausdorff distance. We set $U(t):= \bigcap_{i=1}^p U_i(t)$. Moreover we assume there exist constants $\alpha, \beta, M >0$ such that for all $ t$ in $[0,T]$, $g_i(t,\cdot)$ belongs to $ C^2(U_i(t))$ and satisfies 
\begin{equation}
 \tag{A1}
\forall \, \qqq \in U_i(t) \virg \alpha \leq |\nabla_{\qqq} g_i (t,\qqq) | \leq \beta,
\label{gradg}
\end{equation}
\begin{equation}
 \tag{A2}
\forall \, \qqq \in \R^d \virg|\partial_{t} g_i (t,\qqq) | \leq \beta,
\label{dtg}
\end{equation}
\begin{equation}
 \tag{A3}
\forall \, \qqq \in U_i(t) \virg  |\partial_{t}\nabla_{\qqq} g_i (t,\qqq) | \leq M,
\label{dtgradg}
\end{equation}
and 
\begin{equation}
 \tag{A4}
\forall \, \qqq \in U_i(t) \virg  |\mathrm{D}_\qqq^2 g_i (t,\qqq) | \leq M.
\label{hessg}
\end{equation}
For all $t\in [0,T]$ and $\qqq \in Q(t)$, we denote by $I(t,\qqq)$ the active set at $ \qqq$
\begin{equation}
 I(t,\qqq):=\left\{i \in\I \virg  g_i(t,\qqq)= 0 \right\}
\label{def:I}
\end{equation} 
and for every $\rho >0 $, we put:
\begin{equation}
 I_\rho(t,\qqq):=\left\{i \in\I \virg  g_i(t,\qqq) \leq \rho \right\}.
\label{def:Irho}
\end{equation} 
In addition we assume there exist $\gamma \geq 1 $ and $\rho >0 $ such that for all $t \in [0,T] $,
\begin{equation}
 \tag{A5}
\forall \, \qqq \in Q(t) \virg \forall \, \lambda_{i} \geq 0, \sum_{i \in I_\rho(t,\qqq)} \lambda_{i} | \nabla_\qqq \, g_i(t,\qqq)|  \leq  \gamma\left| \sum_{i \in I_\rho(t,\qqq)}  \lambda_{i} \nabla_\qqq \, g_i (t,\qqq) \right|.
\label{inegtrianginverserho}
\end{equation}
We will use the following weaker assumption:
\begin{equation}
 \tag{A5'}
\forall \, \qqq \in Q(t) \virg \forall \, \lambda_{i} \geq 0, \sum_{i \in I(t,\qqq)} \lambda_{i} | \nabla_\qqq \, g_i(t,\qqq)|  \leq  \gamma\left| \sum_{i \in I(t,\qqq)}  \lambda_{i} \nabla_\qqq \, g_i (t,\qqq) \right|.
\label{inegtrianginverse}
\end{equation}

\mb In particular, this last assumption implies that for all $t $, the gradients of the active inequality constraints $\nabla_\qqq \, g_i (t,\qqq) $ are positive-linearly independent at all $\qqq \in Q(t)$, which is usually called the Mangasarian-Fromowitz constraint qualification (MFCQ). Conversely the MFCQ condition at a point $\qqq $ yields a local version of Inequality (\ref{inegtrianginverse}).

\mb We recall some useful results established in \cite{jvnm} (Propositions 2.8, 2.9, 2.11 and Theorem 2.12 in \cite{jvnm}).

\begin{prpstn}
 For all $t \in [0,T] $ and $\qqq \in Q(t) $, $$\NN(Q(t),\qqq )= \sum \NN(Q_{i}(t),\qqq ) = - \sum_{i \in I(t,\qqq)} \R^+ \nabla_\qqq \, g_i(t,\qqq).$$
 \label{prop:coneprox}
\end{prpstn}

\begin{prpstn}
Under assumptions (\ref{gradg}), (\ref{hessg}) and (\ref{inegtrianginverse}), the set $Q(t)$ is $\eta$-prox-regular with $ \eta = \dsp  \frac{\alpha}{M \gamma} , $ 
for every $t \in [0,T] $.
\label{Qprox}
\end{prpstn}

\begin{prpstn} \label{Qlip} Under assumptions (\ref{Ui}), (\ref{gradg}), (\ref{dtg}) and (\ref{inegtrianginverserho}), the set-valued map $Q$ is Lipschitz continuous with respect to the Hausdorff distance. More precisely there exists $K_L >0 $ such that $$ \forall t,s \in [0,T]  \virg d_H(Q(t),Q(s)) \leq  K_L |t-s |.$$
\end{prpstn}

\begin{thrm}
\label{theo:wp}
Let $T>0 $ and $ \ff :[0,T] \times \R^d \rightarrow \R^d $ be a measurable map satisfying:
\begin{align}
 & \exists K_\ff >0 \virg \forall \qqq,\tilde{\qqq} \in \bigcup_{s \in [0,T]} Q(s) \virg \forall t \in [0,T] \virg  | \ff(t,\qqq) -\ff(t,\tilde{\qqq})| \leq K_\ff |\qqq - \tilde{\qqq}|
\label{lip} \vsp \\
& \exists L_\ff >0 \virg \forall \qqq \in \bigcup_{s \in [0,T]} Q(s) \virg \forall t \in [0,T] \virg | \ff(t,\qqq)|\leq L_\ff(1+|\qqq|).
\label{lingro}
\end{align}
Then, under Assumptions (\ref{Ui}), (\ref{gradg}), (\ref{dtg}), (\ref{hessg}) and (\ref{inegtrianginverserho}) for all $\qqq_0 \in Q(0) $, the
following problem 
\begin{equation}
 \left \{ 
\begin{array}{l}
\displaystyle\frac {d \qqq}{dt}(t) + \NN(Q(t),\qqq(t)) \ni \ff(t,\qqq(t)) \ \textmd{for a.e. } t \in [0,T]  \vsp
\\
\qqq(0)=\qqq_0,
\end{array}
\right.
\label{incldiff2}
\end{equation}
has one and only one absolutely continuous solution $\qqq $ satisfying $ \qqq(t) \in Q(t)$ for every $t \in [0,T] $.
\end{thrm}

\section{Time-stepping scheme}
\label{sec:scheme}
Let us detail the numerical scheme proposed in \cite{jvnm} to approximate the solution
 of~(\ref{incldiff2}) on the time interval $[0, T] $.  Let $ n\in \mathbb{N}^\star$, $h = T/n $ be the time
 step and $t_k^n=k h $ be
 the computational times. 
 We denote by 
$\qqq_k^n $ the approximation of  $\qqq(t_k^n )$ with $\qqq_0^n=\qqq_0 $. The next configuration is computed as follows:
\begin{equation}
 \qqq_{k+1}^n : = \P_{\Qc( t_{k+1}^n ,\qqq_k^n)} (\qqq_k^n  + h \ff(t_k^n, \qqq_k^n ))
\label{schema}
\end{equation} 
with
$$
\Qc(t,\qqq)  := \{ \tilde \qqq \in \R^d \virg g_{i}(t,\qqq) + \langle \nabla_\qqq \, g_{i}(t,\qqq), 
\tilde \qqq - \qqq \rangle \geq 0 \quad \forall \, i 
\} \textmd{ for } \qqq \in U(t):=\bigcap_{i=1}^p U_i(t).
$$
We recall that all the gradients $\nabla_\qqq \, g_{i}(t,\qqq)$ are well-defined provided that $\qqq\in U(t)$. Indeed
it can be checked that this scheme is well-defined, more precisely for $h < \frac{c}{K_L} $
\begin{equation} \Qc(t_{k+1}^n,\qqq_k^n ) \subset Q(t_{k+1}^n)\subset U(t_{k+2}^n) \label{eq:inclusion}
 \end{equation}
 with $c$ and $K_L $ respectively given by Assumption~(\ref{Ui}) and Proposition \ref{Qlip} (see Proposition 3.1 in \cite{jvnm}). Thus every computed configuration is feasible and the set $\Qc(t,\qqq)$ can be seen as an inner convex approximation of $Q(t)$ with respect to $\qqq$. The dependence on $\qqq$ of $\Qc(t,\qqq)$ do not allow us to have recourse to usual techniques (see for instance \cite{Moreau, Monteiro, Thibbv}) based on the time-regularity of the sets $Q(t)$.\\
This scheme is a prediction-correction algorithm: predicted position vector $\qqq_k^n  + h \ff(t_k^n,\qqq_k^n )$, that may not be admissible, is projected onto the approximate set of feasible configurations.

\mb Before stating the result of convergence, we introduce some last notations.
We define the piecewise constant function $ \ff^n $ as follows, 
\be{eq:fn} \ff^n(t)=\ff(t_k^n, \qqq_k^n)
\ \textmd{ if } \ t \in
[t_k^n,t_{k+1}^n [,\ k<n \ \textmd{ and } \ \dsp \ff^n(T)=\ff(t_{n-1}^n, \qqq_{n-1}^n).\ee 
We denote by $\qqq^n$ the continuous, piecewise linear function satisfying for $k \in\{ 0, \dots, n\}$, $\qqq^n(t_k^n)= \qqq_k^n .$
To finish, we introduce the functions $\rho $ and $\theta $ defined by $$\rho^n(t)=t_k^n \textmd{ and }\theta^n(t)=t_{k+1}^n
\textmd{ if } t \in [t_k^n, t_{k+1}^n[,\  \rho^n(T)=T \textmd{ and }\theta^n(T)=T. $$

\mb We recall some results about these approximate solutions (see Subsection 3.2 in \cite{jvnm} for details)~:
\begin{thrm}
\label{theo:qhq} Let us suppose that there exists $H_\ff>0$ for all $\dsp \qq\in \bigcup_{s\in [0,T]} Q(s)$ and $t,s\in[0,T]$ 
\be{ass:holder} |\ff(t,\qq)-\ff(s,\qq)| \leq H_\ff |t-s|^{1/2}. \ee
 Then with the assumptions of Theorem~\ref{theo:wp}, $\qqq^n$ tends to $\qqq$
 in $C^0([0,T], \R^d)$, where $t\mapsto \qqq(t)$ is the unique solution
of~(\ref{incldiff2}).
\end{thrm}

\begin{rmrk} We can replace the definition (\ref{eq:fn}) of $\ff^n$ with
$$ \ff^n(t)= \frac{1}{h} \int_{t_k^n}^{t_{k+1}^n} \ff(s, \qqq_k^n) ds
\ \textmd{ if } \ t \in
[t_k^n,t_{k+1}^n [,\ k<n \ \textmd{ and }\  \dsp \ff^n(T)=\frac{1}{h} \int_{T-h}^{T} \ff(s, \qqq_{n-1}^n) ds.$$
\end{rmrk}

\begin{prpstn} \label{prop:properties}
 There exist $C,D,K>0$ such that
$$ \sup_{n} \|\qqq^n\|_{L^\infty([0,T])} \leq C, \quad \sup_{n} \left\|\frac{d\qqq^n}{dt} \right\|_{L^\infty([0,T])} \leq K $$
and for $n$ large enough,  
\begin{equation} d_{\Qc(t_{k+1}^n, \qqq_k^n)}(\qqq_k^n ) \leq D h.  \label{distQc}
\end{equation}
\end{prpstn}

\mb We now come to the main result of the present paper which specifies the convergence order of the previous scheme.

\begin{thrm} \label{thm:rateconv}
Under (\ref{ass:holder}), there exists a constant $C_0>0$ such that for $n$ large enough
$$ \|\qqq^n-\qqq\|_{L^\infty([0,T])} \leq C_0\left(\frac{T}{n}\right)^{1/2},$$
where $\qqq$ is the solution of (\ref{eq:incldiff}).
\end{thrm}
 
\begin{proof} We check that the sequence $(\qqq^n)_n$ is of Cauchy type. \\ Let $m\geq n$ be large enough. 
Since for $k\in\{0,..,n-1\}$ 
$$ \qqq_{k+1}^n  = \P_{\Qc( t_{k+1}^n ,\qqq_k^n)} (\qqq_k^n  + h \ff(t_k^n, \qqq_k^n ))$$
and $\Qc( t_{k+1}^n ,\qqq_k^n)$ is a closed convex set, it comes~: for all $y\in \R^{d}$

\be{eq:ineg} \ll \qqq_k^n  + h \ff(t_k^n, \qqq_k^n )-\qqq_{k+1}^n,y-\qqq_{k+1}^n\rr \leq |\qqq_k^n  + h \ff(t_k^n, \qqq_k^n )-\qqq_{k+1}^n| d_{\Qc( t_{k+1}^n ,\qqq_k^n)}(y). \ee
By Assumption (\ref{lingro}) and Proposition \ref{prop:properties}, we get for almost $t\in [0,T]$, 
\be{eq:ineg2} \left|\ff^n(t) - \frac{d\qqq^n}{dt}(t)\right|\leq L_\ff(1+C) + K. \ee
Consequently by dividing (\ref{eq:ineg}) by $h$, we obtain for all $y\in \R^d$
\be{toto} -\left\ll \frac{d\qqq^n}{dt}(t) - \ff^n(t),y-\qqq^n(\theta^n(t)) \right\rr \leq C_1 d_{\Qc( \theta^n(t) ,\qqq^n(\rho^n(t))}(y)\ee
with $C_1:=L_f(1+C) + K$. \\
Taking $y=\qqq^m(\theta^m(t))$, it follows
$$ - \left\ll \frac{d\qqq^n}{dt}(t) - \ff^n(t),\qqq^m(\theta^m(t))-\qqq^n(\theta^n(t)) \right\rr \leq C_1 d_{\Qc( \theta^n(t) ,\qqq^n(\rho^n(t))}(\qqq^m(\theta^m(t))).$$

\mb {\bf First case} : $|\qqq^m(\theta^m(t))-\qqq^n(\rho^n(t))|\leq r/8$ (with $r$ later introduced in Theorem \ref{thm:condqualif}). \\
Let us choose $w \in \P_{Q( \theta^n(t))} (\qqq^m(\theta^m(t))) $. Hence $w\in Q( \theta^n(t))$ and 
$$ |w-\qqq^n(\rho^n(t))|\leq |w-\qqq^m(\theta^m(t))| + \frac{r}{8} \leq d_H(Q(\theta^n(t)),Q(\theta^m(t))) + \frac{r}{8} \leq \frac{K_LT}{n}+\frac{r}{8}< \frac{r}{4}$$
by Proposition \ref{Qlip}, for $n$ large enough. Moreover from (\ref{eq:inclusion}) it comes $\qqq^n(\rho^n(t)) \in Q(\rho^n(t)) \subset U(\theta^n(t))$ for $n>T K_L/c$  and then Inequality (\ref{distQc}) implies that
$$ d_{\Qc( \theta^n(t) ,\qqq^n(\rho^n(t))}(\qqq^n(\rho^n(t))) \leq \frac{r}{4}$$
for $n$ large enough. Then, by Theorem \ref{thm:condqualif} and Proposition \ref{prop:d1cont}, we deduce
$$ d_{\Qc( \theta^n(t) ,\qqq^n(\rho^n(t))}(w) \leq \kappa |\qqq^n(\rho^n(t))-w|^2$$
where $\kappa:=\Theta p M/(2\alpha)$. Hence with Propositions \ref{prop:properties} and \ref{Qlip}
\begin{align*}
d_{\Qc( \theta^n(t) ,\qqq^n(\rho^n(t)))}(\qqq^m(\theta^m(t))) & \leq \kappa |\qqq^n(\rho^n(t))-w|^2 + |\qqq^m(\theta^m(t))-w| \\
 & \leq \kappa |\qqq^n(\rho^n(t))-w|^2 + d_H(Q( \theta^n(t)), Q( \theta^m(t) )).
\end{align*}
Since 
\begin{align*}
|\qqq^n(\rho^n(t))-w| & \leq |\qqq^n(\rho^n(t))-\qqq^n(t)| + |\qqq^n(t)-\qqq^m(t)| +|\qqq^m(t)-\qqq^m(\theta^m(t))| + |\qqq^m(\theta^m(t))-w| \\
 & \leq \frac{KT}{n} + |\qqq^n(t)-\qqq^m(t)| + \frac{KT}{m} + \frac{K_L T}{m} \\
 & \leq  |\qqq^n(t)-\qqq^m(t)| + 2\frac{KT}{n} + \frac{K_L T}{n},
\end{align*}
we get
\begin{align*}
d_{\Qc( \theta^n(t) ,\qqq^n(\rho^n(t)))}(\qqq^m(\theta^m(t))) \leq 2\kappa |\qqq^n(t)-\qqq^m(t)|^2 +2\kappa \left(\frac{2KT+K_LT}{n}\right)^2 + \frac{K_LT}{n}.
\end{align*}
Finally,
$$ - \left\ll \frac{d\qqq^n}{dt}(t) - \ff^n(t),\qqq^m(\theta^m(t))-\qqq^n(\theta^n(t)) \right\rr \leq 2 C_1 \kappa |\qqq^n(t)-\qqq^m(t)|^2 + \frac{C_2}{n},$$
with $C_2:=C_1(K_LT+2\kappa T^2(2K+K_L)^2)$. 

\mb {\bf Second case} : $|\qqq^m(\theta^m(t))-\qqq^n(\rho^n(t))|\geq r/8$. \\
Then by (\ref{toto}), 
\begin{align*} 
- \left\ll \frac{d\qqq^n}{dt}(t) - \ff^n(t),\qqq^m(\theta^m(t))-\qqq^n(\theta^n(t)) \right\rr & \leq C_1 |\qqq^m(\theta^m(t))-\qqq^n(\theta^n(t))| \\
& \leq C_1 |\qqq^m(\theta^m(t))-\qqq^n(\rho^n(t))| + C_1\frac{KT}{n} \\
& \hspace{-2cm} \leq  \frac{8}{r}C_1 |\qqq^m(\theta^m(t))-\qqq^n(\rho^n(t))|^2+C_1\frac{KT}{n} \\
 & \hspace{-2cm} \leq \frac{16}{r} C_1  |\qqq^m(t)-\qqq^n(t)|^2 + \frac{64}{r}C_1\left(\frac{KT}{n}\right)^2 + C_1\frac{KT}{n}.
\end{align*}

\mb {\bf End of the proof} : 

\mb
By setting $C_3:=\max\{2C_1\kappa,16C_1/r\}$ and $C_4:=\max\{C_2,\frac{64C_1(KT)^2}{r}+C_1KT\}$, we get
$$ - \left\ll \frac{d\qqq^n}{dt} (t) - \ff^n(t),\qqq^m(\theta^m(t))-\qqq^n(\theta^n(t)) \right\rr \leq C_3 |\qqq^m(t)-\qqq^n(t)|^2 + \frac{C_4}{n}. $$
Therefore
$$ - \left\ll \frac{d\qqq^n}{dt} (t) - \ff^n(t),\qqq^m(t)-\qqq^n(t) \right\rr \leq C_3 |\qqq^m(t)-\qqq^n(t)|^2 + \frac{C_5}{n}, $$
with $C_5:=C_4+2C_1 KT$.
By summing the previous inequality and the other one obtained by changing the role of $n$ and $m$, it yields
$$ \left\ll \frac{d\qqq^m}{dt}(t) - \frac{d\qqq^n}{dt}(t),\qqq^m(t)-\qqq^n(t) \right\rr \leq 2C_3 |\qqq^m(t)-\qqq^n(t)|^2 + \frac{2C_5}{n}+|\ff^n(t)-\ff^m(t)||\qqq^m(t)-\qqq^n(t)|.$$
Furthermore by (\ref{lip}) and (\ref{ass:holder})
\begin{align*} 
|\ff^n(t)-\ff^m(t)| & \leq |\ff(\rho^m(t),\qqq^m(\rho^m(t)))-\ff(\rho^m(t),\qqq^n(\rho^n(t)))| + |\ff(\rho^m(t),\qqq^n(\rho^n(t)))-\ff(\rho^n(t),\qqq^n(\rho^n(t)))| \\
 & \leq K_\ff|\qqq^m(\rho^m(t))-\qqq^n(\rho^n(t))|+H_\ff |\rho^n(t)-\rho^m(t)|^{\frac{1}{2}} \\
 & \leq K_\ff|\qqq^m(t)-\qqq^n(t)| + 2K_\ff\frac{KT}{n} + H_\ff |\rho^n(t)-\rho^m(t)|^\frac{1}{2},
\end{align*}
and so
\begin{align*}
 |\ff^n(t)-\ff^m(t)||\qqq^m(t)-\qqq^n(t)| & \leq K_\ff|\qqq^m(t)-\qqq^n(t)|^2 + \left(2K_\ff\frac{KT}{n} + H_\ff \left(\frac{T}{n}\right)^{\frac{1}{2}}\right)|\qqq^m(t)-\qqq^n(t)| \\
 & \leq (\frac{1}{2}+K_\ff)|\qqq^m(t)-\qqq^n(t)|^2 + \frac{1}{2} \left(2K_\ff KT + H_\ff \sqrt{T}\right)^2 \frac{1}{n}. 
\end{align*}
Finally,
$$ \left\ll \frac{d\qqq^m}{dt}(t) - \frac{d\qqq^n}{dt}(t),\qqq^m(t)-\qqq^n(t) \right\rr \leq (2C_3+\frac{1}{2}+K_\ff)|\qqq^m(t)-\qqq^n(t)|^2 + \frac{C_6}{n},$$
with $C_6:=2C_5 + \frac{1}{2}\left(2K_\ff KT + H_\ff \sqrt{T}\right)^2$. \\
By applying Gronwall's Lemma, we have
$$ \|\qqq^m-\qqq^n\|_{L^\infty([0,T])}\leq \left(\frac{2C_6}{n}\right)^{1/2} \exp((2C_3 +\frac{1}{2}+ K_\ff)T).$$
Then, we conclude the proof by taking the limit for $m\to \infty$.
\end{proof}

\begin{rmrk} \label{rem:hilbert} This proof allows us to get around the compactness arguments employed in \cite{jvnm} to obtain the convergence of $\qqq_h$. Consequently, this result can be extended to the Hilbertian case. Then it can be checked that the limit satisfies the differential inclusion (\ref{eq:incldiff}) by following the same reasoning as in \cite{jvnm}. 
\end{rmrk}

\mb It remains to prove Proposition \ref{prop:d1cont} and Theorem \ref{thm:condqualif}. We now check the first result whereas the second one will be established in the next section.

\begin{prpstn} \label{prop:d1cont}
For $i\in\{1,...,p\}$ and $\qqq\in U(t)$ we set
$$ \Qc_i(t,\qqq):= \{ \tilde \qqq \in \R^d \virg g_{i}(t,\qqq) + \langle \nabla_\qqq \, g_{i}(t,\qqq), 
\tilde \qqq - \qqq \rangle \geq 0 \}.$$
Then, for all $t\in[0,T]$, all $\qq \in Q(t)$ and all $\tilde{\qq} \in U(t)$, we have for all $i\in\{1,...,p\}$
\begin{equation} d_{\Qc_i(t,\tilde{\qq})}(\qqq) \leq \frac{M}{2\alpha} |\qqq-\tilde{\qq}|^2. \label{eq:dist} \end{equation}
\end{prpstn}

\begin{proof}
Let consider $i\in\{1,...,p\}$, $\tilde{\qq}\in U(t)$ and $ \qqq\in Q(t) \subset Q_i(t)$. We assume that $\qqq \notin \Qc_i(t,\tilde{\qq})$ (otherwise (\ref{eq:dist}) obviously holds). \\
For $\ell \geq 0$, we define
$$ z(\ell):= \qqq + \ell \nabla g_i(t,\tilde{\qq}).$$
The point $z(\ell)$ belongs to $\Qc_i(t,\tilde{\qq})$ if and only if 
$$ g_i(t,\tilde{\qq})+\ll \nabla g_i(t,\tilde{\qq}), \qqq-\tilde{\qq}\rr + \ell |\nabla g_i(t,\tilde{\qq})|^2 \geq 0,$$
which is equivalent to
$$ \ell\geq \ell_0:= -\frac{g_i(t,\tilde{\qq})+\ll \nabla g_i(t,\tilde{\qq}), \qqq-\tilde{\qq}\rr}{|\nabla g_i(t,\tilde{\qq})|^2} \geq 0.$$
Thus, 
\begin{align*}
 d_{\Qc_i(t,\tilde{\qq})}(\qqq) & \leq |\qqq-z(\ell_0)| \leq \ell_0 |\nabla g_i(t,\tilde{\qq})| \\
 & \leq -\frac{g_i(t,\tilde{\qq})+\ll \nabla g_i(t,\tilde{\qq}), \qqq-\tilde{\qq}\rr}{|\nabla g_i(t,\tilde{\qq})|} \\
 & \leq \frac{1}{|\nabla g_i(t,\tilde{\qq})|} \int_0^1 s \mathrm{D}_\qqq^2 g_i(t,\tilde{\qq}+s(\qqq-\tilde{\qq}))(\qqq-\tilde{\qq},\qqq-\tilde{\qq}) ds,
\end{align*}
because $g_i(t,\qqq)\geq 0$. We conclude to (\ref{eq:dist}) by Assumptions (\ref{gradg}) and (\ref{hessg}).
\end{proof}

\section{Metric qualification condition}
\label{sec:mqc}

This section is devoted to the proof of Theorem \ref{thm:condqualif}, which corresponds to a metric qualification condition for the sets $\Qc_i$. Aiming that, we recall some notions of subdifferential calculus.

\begin{dfntn}[proximal subdifferential]
 \label{def:gprox}
Let $f : \R^d \rightarrow \overline{\R}$ be a lower semicontinuous function which is finite at $x \in \R^d$.
The proximal subdifferential of $f$ at $x$ is defined by 
$$\partial^P f(x):=\left\{ x^\star \in \R^d, \ \exists \alpha,\beta>0 ,\ \forall |h|\leq \beta, \ f(x+h)-f(x) \geq \ll x^\star , h \rr -\alpha |h|^2\right\}. $$ 
\end{dfntn}

\begin{dfntn}[limiting subdifferential]
 \label{def:glimiting}
Let $f : \R^d \rightarrow \overline{\R}$ be a lower semicontinuous function which is finite at $x \in \R^d$.
The limiting (or Mordukhovich) subdifferential of $f$ at $x$ is defined by 
$$\partial^L f (x):=\left\{ x^\star \in \R^d, \ x^\star =\lim_{k\to \infty}  x_k^\star \quad \textrm{with} \quad  x_k^\star \in \partial^Pf(x_k), \ x_k \to x \quad \textrm{and} \quad  \ f(x_k) \to f(x)  \right\}. $$ 
\end{dfntn}

\begin{dfntn}[Clarke subdifferential]
 \label{def:gclarke}
Let $f : \R^d \rightarrow \overline{\R}$ be a Lipschitz continuous function.
The Clarke subdifferential $\partial^C f(x)$ of $f$ at $x$ can be defined (see \cite{Clarke}) as the closed convex hull of the limiting subdifferential~: 
$$\partial^C f(x):= \overline{ \textrm{conv}\  \partial^L f (x) }. $$
\end{dfntn}

\mb This notion has been extended for less regular functions, we refer the reader to \cite{Rockafellar} for details.

\mb The following property is a special case of the exact sum rule for the Clarke subdifferential (see Theorem 2 of \cite{Rockafellar}):

\begin{lmm}[Optimality property] \label{lem:opt} Let $f:\R^d \rightarrow \R\cup \{+\infty\}$ be a lower semicontinuous function and $\phi:\R^d \rightarrow \R$ a convex Lipschitz function. If $x\in\R^d$ is a finite local minimum of $f+\phi$ then
$$ 0 \in \partial^C f(x) + \partial^C \phi(x).$$
\end{lmm}

\mb Let us recall the variational principle of Ekeland (see \cite{Ekeland}).

\begin{prpstn}[Ekeland variational principle] \label{prop:ekeland} Let $f:\R^d \rightarrow \R\cup \{+\infty\}$ be a lower semi-continuous function which is bounded from below. Let $\epsilon>0$ and $x\in \R^d$ such that 
$$ \inf f \leq f(x) \leq \inf f + \epsilon.$$
Then for all $\lambda>0$, there exists $w\in\R^d$ satisfying
\begin{itemize}
 \item $f(w)\leq f(x)$ 
 \item $|x-w|\leq \lambda$
 \item for all $z\neq w$, $f(z)>f(w)-\frac{\epsilon}{\lambda} |z-w|$.
\end{itemize}
\end{prpstn}

\mb
The following result comes from Theorem 2.1 in \cite{Ioffe-Outrata}. For an easy reference, we detail the proof.

\begin{lmm} \label{lem:ioffe} Let $f:\R^d\rightarrow \R^+ \cup\{+\infty\}$ be a lower semi-continuous function and $x_0$ with $f(x_0)=0$. Assume there exist $\gamma,\delta>0$ such that for all 
$$x^\star \in \bigcup_{\genfrac{}{}{0pt}{}{x\in B(x_0,2\delta)}{f(x)>0}} \partial^C f(x)$$
we have $|x^\star|\geq \gamma$. Then for all $x\in B(x_0,\delta)$, $d_{\{f=0\}}(x)\leq \gamma^{-1} f(x)$. 
\end{lmm}

\begin{proof} Let $x\in B(x_0,\delta)$. If $f(x)\geq \gamma \delta$, then 
$$d_{\{f=0\}}(x) \leq |x-x_0| \leq \delta \leq \gamma^{-1} f(x).$$
Now, we assume that $0<f(x)< \gamma \delta$ and we set $\epsilon:= f(x)$. Applying the variational principle of Ekeland (see Proposition \ref{prop:ekeland}) to $f$ with $\epsilon$ and any $\lambda\in ]\gamma^{-1}\epsilon,\delta[$. There exists $w=w(\lambda)\in\R^d$ such that $f(w)\leq f(x)$, $|x-w|\leq \lambda$ and 
$$ \forall z\neq w, \quad f(z)>f(w)-\frac{\epsilon}{\lambda} |z-w|.$$
Consequently, $w$ minimizes $f + \epsilon \lambda^{-1}|\cdot-w|$ and by Lemma \ref{lem:opt} it comes 
$$ 0\in \partial^C f(w) + \partial^C \psi(w)$$
where $\psi(\cdot)= \epsilon \lambda^{-1} |\cdot-w|$. \\
So there exists $x^\star \in \partial^C f(w)$ with $|x^\star|\leq \epsilon\lambda^{-1} <\gamma$. If $f(w)>0$, that is in contradiction with the assumptions as $|w-x_0|\leq |w-x|+|x-x_0|\leq 2\delta$ and so we deduce that necessarily $f(w)=0$. Then we conclude to the desired result, since
$$d_{\{f=0\}}(x) \leq |x-w|\leq \lambda$$
holds for every $\lambda\in ]\gamma^{-1}\epsilon,\delta[$.
\end{proof}

\mb From now on, we come back to the framework of the previous sections and prove the metric qualification condition of sets $\Qc_i$.

\mb
In the sequel, we introduce convex sets $C_i $ for $i \in \{1,...,p\}$ and their intersection $ \dsp C=\bigcap_{i=1}^p C_i.$
We consider the following set-valued map $F$ 
\begin{equation} F: \left\{ \begin{array}{lcl} \R^d &\rightrightarrows &\R^{dp} \\ x &\mapsto & F(x):=(C_1-x)\times\dots \times (C_p-x). \end{array} \right. \label{def:F}\end{equation}
Let us note that $0\in F(x)$ if and only if $x\in C$. 

\begin{prpstn} \label{prop:sousdif} Consider the function $f$ defined by $f(x):= d_{F(x)}(0)$ where $F$ is given by (\ref{def:F}). The map $f$ is Lipschitz continuous and for all $x\notin C$,
$$ \partial^P f(x) \subset \partial^C f(x) = \left\{ \sum_{i,\ x\notin C_i} \frac{y_i}{|y|} \right\},$$
where $y = \P_{F(x)}(0)$. In other words, for all $i\in\{1,...,p\}$, $y_i+x\in \P_{C_i}(x)$, hence  $-y_i\in \NN(C_i,x+y_i)$.
 \end{prpstn}

\begin{proof} For all $x\in\R^d$, 
$$ f(x) = d_{F(x)}(0)=d_\Pi(\phi(x))$$
where $\Pi:=\otimes_{i=1}^p C_i$ and $\phi(x):=(x,\dots, x)\in\R^{dp}$.
For $x\notin C$,
$$ \partial^C f(x) = \partial^C (d_{\Pi} \circ \phi) (x) =\phantom{}^t (1,\dots, 1) \cdot \partial^C d_{\Pi} (\phi(x))$$
thanks to Corollary 1 in \cite{Rockafellar}.
By convexity of the sets $C_i$, $d_\Pi$ is a convex function and so 
$$ \partial^C f(x) = \phantom{}^t (1,\dots, 1) \cdot  \partial^P d_{\Pi} (\phi(x)),$$
see Remark \ref{rem:interesting}.
First we claim that
\begin{equation} \dsp  \partial^P d_{\Pi} (\phi(x)) \subset \left[\otimes_{i=1}^p {\mathcal E}_i(x)  \right] \bigcap S(0,1), \label{eq:cone} \end{equation}
with ${\mathcal E}_i(x):=\frac{d_{C_i}(x)}{d_\Pi(\phi(x))} \partial^P d_{C_i} (x)$ if $x\notin C_i$ and ${\mathcal E}_i(x):=\{0\}$ else. \\
Indeed, let $x^\star$ belong to $\partial^P d_{\Pi}(\phi(x))$. By definition, for some $\alpha>0$ and for all small enough $h\in \R^{dp}$,
$$d_{\Pi}(\phi(x)+h)-d_{\Pi}(\phi(x)) \geq \ll x^\star , h \rr -\alpha |h|^2.$$
Let us fix an index $i\in\{1,\dots , p\}$ and set $h=(0,\dots,0,h_i,0,\dots,0)$. Since $d_\Pi(\phi(\cdot))^2=\sum_i d_{C_i}^2$, it follows that for all small enough $h_i\in \R^d$
$$\sqrt{d_{\Pi}(\phi(x))^2 + d_{C_i}(x+h_i)^2-d_{C_i}(x)^2} -\sqrt{d_{\Pi}(\phi(x))^2} \geq \ll x^\star_i , h_i \rr -\alpha |h_i|^2.$$
By a first order expansion, we get
$$ \frac{d_{C_i}(x+h_i)^2-d_{C_i}(x)^2}{2d_{\Pi}(\phi(x))} \geq \ll x^\star_i , h_i \rr -\alpha' |h_i|^2,$$
with another numerical constant $\alpha'$. Then, we obtain with another constant $\alpha''$ and for all small enough $h_i\in \R^d$
$$ \frac{d_{C_i}(x)}{d_{\Pi}(\phi(x))} \left(d_{C_i}(x+h_i)-d_{C_i}(x) \right) \geq \ll x^\star_i , h_i \rr -\alpha'' |h_i|^2.$$
If $x\in C_i$ then $d_{C_i}(x)=0$ and so we deduce that $x^\star_i=0$. If $x\notin C_i$ then by definition of the proximal normal cone, 
$$ \frac{d_{\Pi}(\phi(x))}{d_{C_i}(x)} x^\star_i \in \partial^P d_{C_i}(x) \subset S(0,1),$$
see Remark \ref{rem:interesting}. \\
So $|x^\star_i|= d_{C_i}(x) d_\Pi(\phi(x))^{-1}$ and so $|x^\star|=1$, which concludes the proof of (\ref{eq:cone}). \\
Let us now finish the proof of the proposition. Thus
$$ \partial^C f(x) \subset \sum_{i,\ x\notin C_i} {\mathcal E}_i(x) \subset  \sum_{i,\ x\notin C_i} \frac{d_{C_i}(x)}{d_\Pi(\phi(x))} \partial^P d_{C_i}(x).$$
We set $z=(z_1,...,z_p)\in \R^{dp}$ with for all $i$, $z_i=\P_{C_i}(x)$ or equivalently $z=\P_{\Pi}(\phi(x))$. By Theorem 1.105 in \cite{Mordu}, 
$$ \partial^P d_{C_i}(x) \subset \partial^P d_{C_i}(z_i) \cap S(0,1) = \left\{ \frac{x-z_i}{|x-z_i|}\right\}.$$
Consequently, we have
$$ \partial^C f(x) \subset \sum_{i,\ x\notin C_i} {\mathcal E}_i(x) \subset  \left\{ \sum_{i,\ x\notin C_i} \frac{d_{C_i}(x)}{d_\Pi(\phi(x))} \frac{x-z_i}{|x-z_i|}  \right\} = \left\{ \sum_{i,\ x\notin C_i} \frac{x-z_i}{|\phi(x)-z|}  \right\}.$$
We finish the proof by choosing $y:=\phi(x)-z\in \R^{dp}$. 
\end{proof}

\begin{rmrk} \label{rem:interesting} Let $S\subset \R^d$ be a closed convex set and $x\notin S$, then $\partial^P d_S(x) =\partial^C d_S(x) \subset S(0,1)$. Indeed with $w:= \P_{S}(x)$ and vectors $h=\epsilon(w-x)$ for small enough $\epsilon$, we remark that $d_S(x+\epsilon(w-x)) = d_S(x)-\epsilon|w-x|$. Hence, by Definition \ref{def:gprox}, we obtain for every $x^\star \in \partial^P d_S(x)$
$$ -|h| \geq \ll x^\star, h \rr - \alpha |h|^2.$$
By dividing by $|h|$ and letting $\epsilon$ go to $0$, we deduce that $|x^\star|\geq 1$. We also conclude to $|x^\star|=1$ since $d_S$ is $1$-Lipschitz.
\end{rmrk}

\begin{thrm}
 There exist $r$ and $\Theta $ such that for all $t \in [0,T]$, for all $\tilde{\qq} \in U(t)$ satisfying $d_{\Qc(t,\tilde{\qq})}(\tilde{\qq})\leq r/4$ and all $\qq \in B(\tilde{\qq}, r/4)$, 
$$d_{\Qc(t,\tilde{\qq})}(\qq)  \leq \Theta \sum_{i=1} ^p d_{\Qc_i(t,\tilde{\qq})}(\qq).$$
Indeed we can choose $\Theta=\frac{2\gamma \beta}{\alpha}$ and $r=\min(\frac{4\rho}{13\beta},\frac{\alpha}{2M\gamma})$.
\label{thm:condqualif}
\end{thrm}

\begin{proof} Consider $r=\min(\frac{4\rho}{13\beta},\frac{\alpha}{2M\gamma})$.
Let us fix  $t \in [0,T]$ and $\tilde{\qq} \in U(t)$ satisfying $d_{\Qc(t,\tilde{\qq})}(\tilde{\qq})\leq r/4$. Consequently there exists $\qq_1\in B(\tilde{\qq},r/4)$ such that $\qq_1\in \Qc(t,\tilde{\qq})$. We define a Lipschitz map $f:= d_{F(\cdot)}(0)$ where $F$ is given by (\ref{def:F}) with $C_i= \Qc_i(t,\tilde{\qq})$.
First we check the assumptions of Lemma \ref{lem:ioffe} for the function $f$ with $x_0=\qq_1$. Indeed $f(\qq_1)=0$ because $\qq_1\in  \Qc(t,\tilde{\qq})$. \\
Let us consider $\qq \in B(\qq_1, r) \cap \Qc(t,\tilde{\qq})^c$, so $\qq\in B(\tilde{\qq},2r)$.
By Proposition \ref{prop:sousdif}, $ \partial^C f(\qq) = \{ \qq^\star\}$ where  $$\qq^\star:= \sum_{i,\ \qq \notin \Qc_i(t,\tilde{\qq})} \pp_i^\star $$
with $\pp^\star = \pp/|\pp|$ and $ \pp =\P_{F(\qq)}(0)$.
Moreover for $i$ satisfying $\qq \notin \Qc_i(t,\tilde{\qq})$, $-\pp_i^\star \in \NN(C_i, \qq +\pp_{i} )$.
Let us define
$$ J(t,\qq):=\left\{j,\ g_j(t,\tilde{\qq})+\ll \nabla g_j(t,\tilde{\qq}),\qq-\tilde{\qq}\rr <0\right\}=\{j,\ \qq\notin \Qc_j(t,\tilde{\qq})\}.$$
It is well-known that there also exist nonnegative reals $(\lambda_i)_{i\in J(t,\qq)}$ satisfying
$$ \qq^\star = \sum_{i\in J(t,\qq)}  \lambda_i \nabla g_i(t,\tilde{\qq}).$$
Hence by Assumption (\ref{hessg})
\begin{align*}
 |\qq^\star| & = \left| \sum_{i \in J(t,\qq) } \lambda_i \nabla g_i(t,\tilde{\qq}) \right| \\
& \geq \left| \sum_{i \in J(t,\qq) } \lambda_i \nabla g_i(t,\qq_1) \right| -  \frac{Mr}{4}  \sum_{i \in J(t,\qq) } \lambda_i. 
\end{align*}
Since $ \qq+\pp_i \in \P_{C_i} (\qq)$, Proposition \ref{prop:d1cont} yields $|\pp_i |= d_{C_i}(\qq) \leq \frac{M}{2\alpha}(2r)^2 \leq r $. Moreover for all $i\in J(t,\qq)$, $\qq+\pp_i\in \partial \Qc_i(t,\tilde{\qq})$ so we have by Assumption (\ref{gradg}) with the convexity of $g_i$
\begin{align*} 
g_i(t,\tilde{\qq}) & \leq -\ll \nabla g_i(t,\tilde{\qq}),\qq+\pp_i-\tilde{\qq}\rr\leq \beta (|\qq-\tilde{\qq}|+|\pp_i|) \\
 & \leq 3\beta r.
\end{align*}
Hence by Assumption (\ref{gradg}),
$$ g_i(t,\qq_1) \leq 3\beta r +\beta r/4=\frac{13}{4} \beta r.$$
Due to the choice of $r$, we deduce that $13\beta r \leq 4\rho$ and thus $J(t,\qq)\subset I_\rho(t,\qq_1)$. From Assumptions (\ref{gradg}), (\ref{hessg}) and (\ref{inegtrianginverserho}), we deduce that
\begin{align*}
 |\qq^\star| & \geq  \gamma^{-1} \sum_{i \in J(t,\qq) } \lambda_i \left|\nabla g_i(t,\qq_1) \right| - \frac{Mr}{4}  \sum_{i \in J(t,\qq) } \lambda_i \\
& \geq ( \frac{\alpha}{\gamma} - \frac{M r}{4}) \sum_{i \in J(t,\qq) } \lambda_i  \\
& \geq \frac{\alpha} {2 \gamma} \sum_{i \in J(t,\qq) } \lambda_i   \\
& \geq \frac{\alpha} {2 \gamma \beta} \sum_{i \in J(t,\qq) } \lambda_i \left|\nabla g_i(t,\tilde{\qq}) \right| \\
& \geq \frac{\alpha} {2 \gamma \beta} \sum_{i \in J(t,\qq) } \left|\pp_i^\star \right| \\
& \geq \frac{\alpha} {2 \gamma \beta} |\pp^\star|= \frac{\alpha} {2 \gamma \beta}.
\end{align*}
We can also apply Lemma \ref{lem:ioffe} and we obtain that for all $\qq\in B(\qq_1,r/2) \supset B(\tilde{\qq},r/4)$
$$ d_{\Qc(t,\tilde{\qq})}(\qq) \leq \Theta \left(\sum_{i\in J(t,\qq)} d_{\Qc_i(t,\tilde{\qq})}(\qq)^2 \right)^{1/2} \leq \Theta \sum_{i=1}^p d_{\Qc_i(t,\tilde{\qq})}(\qq).$$
\end{proof}

\section{Numerical simulations}

The aim of this section is to illustrate the convergence order with an example (due to the modelling of crowd motion in emergency evacuation). We refer the reader to~\cite{TheseJu, MV2} for a complete and detailed description of this model.

\mb We quickly recall the model. It handles contacts, in order to deal with local interactions between people and to describe the whole dynamics of the pedestrian traffic. This microscopic model for crowd motion (where people are identified to rigid disks) rests on two principles. On the one hand, each individual has a spontaneous velocity that he would like to have in the absence of other people. On the other hand, the actual velocity must take into account congestion. Those two principles lead to define the actual velocity field as the Euclidean projection of the spontaneous velocity over the set of admissible velocities (regarding the non-overlapping constraints between disks).

\mb More precisely, we consider $N$ persons identified to rigid disks. For convenience, the disks are supposed here to have the same radius $r$. The
center of the $i$-th disk is denoted by $\qq_i\in \R^2$. Since overlapping is forbidden, the  
vector of positions $\qqq=(\qq_1,..,\qq_N) \in
\R^{2N}$ has to belong to the ``set of feasible configurations'', defined by 
\be{eq:Q} Q:=\left\{ \qqq  \in
\R^{2N},\ D_{ij}( \qqq) \geq 0\quad   \forall \,i \neq j  \right\},  \ee 
where $ D_{ij}(\qqq)=|\qq_i-\qq_j|-2r $ is the signed distance between disks $i$ and $j$. If the global spontaneous velocity of the crowd is denoted by $\UU(\qqq)=(U_1(\qq_1),..,U_N(\qq_N)) \in \R^{2N}$, the previous crowd motion model can be described by the following differential inclusion:
\be{eq:model3} \frac{d\qqq}{dt} + \NN(Q,\qqq) \ni \UU(\qqq).\ee

This evolution problem fits into the theoretical framework developed in this paper. 

For the numerical simulations, we treat an emergency evacuation of a square room (10x10) initially containing $N=150$ persons (identified to rigid disks of radius $r=0.2$). Since the exact solution is unknown, we approach the error as follows
\be{approx} \|\qqq-\qqq_h\|_{L^\infty} \simeq \max_{i=1,...,10} |\qqq_{h_{min}}(t_i)-\qqq_h(t_i)| := e_h \ee
where $t_i:=i T/10$. We choose $h_{min}:=0.01$ and $h\in\{0.02,0.025,0.04,0.05,0.0625,0.08,0.1,0.2,0.5\}$.

\begin{figure}[ht!]
\centering
\psfrag{a}[c]{$\log(h)$}
\psfrag{b}[c]{$\log(e_h)$}
\includegraphics[width=0.8\textwidth]{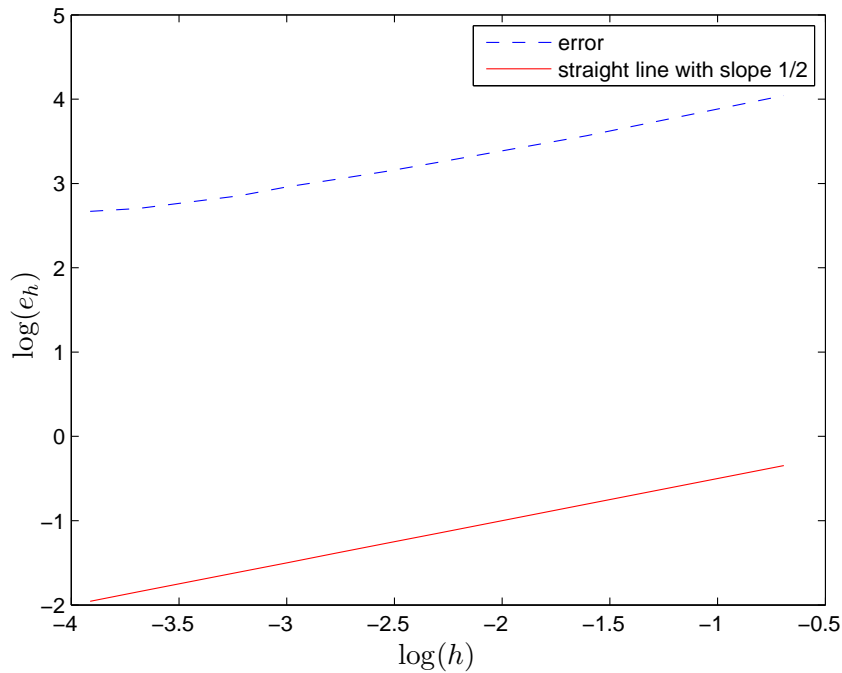}
\caption{Log-log plot of the error $e_h$ versus $h$.}
\label{figure1}
\end{figure}

We observe in Figure \ref{figure1} that the empirical order of convergence is $\frac{1}{2}$. Note that for a time step $h$ close to $h_{min}$, the approximation (\ref{approx}) is not valid.

\end{document}